\documentclass{amsart}

\usepackage{kpfonts}
\usepackage{amsrefs}
\usepackage{tikz-cd}

\newtheorem{theorem}{Theorem}[section]
\newtheorem{lemma}[theorem]{Lemma}
\newtheorem{proposition}[theorem]{Proposition}
\newtheorem{corollary}[theorem]{Corollary}

\theoremstyle{definition}
\newtheorem{definition}[theorem]{Definition}
\newtheorem{example}[theorem]{Example}

\newtheorem{conjecture}[theorem]{Conjecture}

\theoremstyle{remark}
\newtheorem{remark}[theorem]{Remark}

\numberwithin{equation}{section}

\newcommand{\abs}[1]{\lvert#1\rvert}

\newcommand*{\triple}[2][.1ex]{%
  \mathrel{\vcenter{\offinterlineskip%
  \hbox{$#2$}\vskip#1\hbox{$#2$}\vskip#1\hbox{$#2$}}}}
\newcommand*{\triplerightarrow}{\triple{\rightarrow}}

\begin{document}

\title{Ulrich Trichotomy on Del Pezzo Surfaces}

\author{Emre Coskun}
\address{Department of Mathematics, Middle East Technical University, Ankara, TURKEY}
\email{emcoskun@metu.edu.tr}
\thanks{The first author was supported by TUBITAK Project 119F024.}

\author{Ozhan Genc}
\address{Department of Mathematics and Informatics, Jagiellonian University,
{\L}ojasiewicza 6, 30-348, Krak{\'o}w, Poland}
\email{ozhangenc@gmail.com}
\thanks{The second author is supported by Narodowe Centrum Nauki 2018/30/E/ST1/00530.}


\date{\today}



\begin{abstract}
In this article, we use a correspondence between Ulrich bundles on a projective variety and quiver representations to prove that certain del Pezzo surfaces satisfy the Ulrich trichotomy, for any given polarization.
\end{abstract}

\maketitle
\section{Introduction}
An \emph{Ulrich bundle} $\mathcal{E}$ on a polarized variety $(X,H)$ is an ACM (\emph{arithmetically Cohen-Macaulay}) vector bundle whose associated module $\bigoplus_{t} \mbox{H}^0 (X,\mathcal{E}(tH))$ has the maximum possible number of generators, which is $\mbox{deg}(X)\mbox{rank}(\mathcal{E})$. Ulrich bundles on projective varieties have been the subject of increased attention in the last decade. A vector bundle One of the most important questions is the question of existence of Ulrich bundles. This question was posed in \cite{ESW2003}. Since then, the existence problem was solved for many varieties; conversely, no variety that does not have an Ulrich bundle on it is known. Another way to consider the existence problem of Ulrich bundles on varieties is through the \emph{Ulrich representation type} of varieties. The \emph{Ulrich representation type} of a variety is a way to describe the behavior of families of Ulrich bundles. (See Definition \ref{definition:ulrich_representation_type}.)

An analogous problem is the representation type of quivers. Quivers can be divided into \emph{finite type}, \emph{tame}, and \emph{wild} according to the behavior of their representations. (See \cite{Kirillov2016}. We note that he considers finite type quivers as a special case of tame quivers.) The classification of finite type quivers is completely given by the well-known theorem of Gabriel. (\cite{Kirillov2016}*{Theorem 3.3}) It is also possible to give a complete description of tame and wild quivers. (\cite{Kirillov2016}*{Theorem 7.47})

It is therefore an interesting question to ask whether varieties can similarly be classified as varieties of \emph{finite Ulrich representation type}, \emph{tame Ulrich representation type}, and \emph{wild Ulrich representation type}, and whether this trichotomy is exhaustive. (See Definition \ref{definition:ulrich_trichotomy}.) As a matter of fact, the same problem has been already stated for ACM bundles. (\cite{Miro-Roig2018}) However, in the literature, almost all wild ACM type examples have been obtained by showing the existence of Ulrich bundles. So it is natural to ask the same question for the Ulrich representation type.

It must be emphasized that the definition of an Ulrich bundle on a variety $X$ depends on the polarization, i.e. the embedding of $X$ in projective space. It is therefore not a surprise that the Ulrich representation type may depend on the polarization. An example is given by the projective plane $\mathbb{P}^2$. For polarizations given by $H$ and $2H$, where $H$ is the hyperplane class, there is a unique irreducible Ulrich bundle on $\mathbb{P}^2$. (\cite{CG2017}*{Proposition 3.1 and Theorem 5.4}) Hence, $\mathbb{P}^2$ is of finite Ulrich representation type for these two polarizations. For all other polarizations, it is of wild Ulrich representation type. (\cite{CG2017}*{Theorem 6.1 and Theorem 6.2}, \cite{CMR2018}*{Theorem 3}) Even though it is known that del Pezzo surfaces with the regular embedding are Ulrich wild and hence the Ulrich trichotomy is satisfied (\cite{Miro-Roig2018}*{Theorem 3.3.31}), the Ulrich trichotomy problem is open for del Pezzo surfaces with arbitrary polarization.

In this article, we use a theorem of Bondal (\cite{Bondal1989}*{Theorem 6.2}, \cite{Maiorana2017}*{Theorem 2.30}) relating elements of the derived category $\mbox{D}(X)$ of a variety $X$ to representations of a quiver to prove that \emph{Ulrich trichotomy} holds for certain del Pezzo surfaces for all possible polarizations.

The approach we follow in this article supposes the existence of some exceptional collections on the del Pezzo surfaces that we consider. First, it always is the case that there is a line bundle on an extremity of these collections. This implies that, with a suitable shift, we can always suppose this line bundle to be that given by the polarization itself. This has the consequence that the quiver representation given by Bondal's theorem (Theorem \ref{theorem:bondal}) has a zero-dimensional vector space in a vertex at an extremity. Second, the remaining quiver must be hyperbolic; this then allows us to use some results in Kac's article (\cite{Kac1980}). These restrictions allows us to prove that Ulrich trichotomy holds, with any polarization, for $\mathbb{P}^2$, $\mathbb{P}^1 \times \mathbb{P}^1$, the sextic del Pezzo surface $X_3$, and the quintic del Pezzo surface $X_4$. (The subscripts indicate the number of points that are blown up in $\mathbb{P}^2$.) The result for $\mathbb{P}^2$ follows from the previous articles \cite{CG2017} and \cite{CMR2018}, but we include a new proof here for illustrating the method. There are also partial results for $\mathbb{P}^1 \times \mathbb{P}^1$. (\cite{Antonelli2018}) Finally, by \cite{Miro-Roig2018}, it is known that del Pezzo surfaces with the classical embedding are of wild representation type. The other results are completely new. To our knowledge, 
the relation that Bondal's theorem gives between Ulrich bundles and quiver representations has been used in only one article (\cite{LP2017}) where it was shown that the moduli space of Ulrich bundles on the Fano 3-fold $V_5$ is isomorphic to an open subset of the moduli space of stable representations of a certain quiver. We anticipate that this method will be used for other questions about Ulrich bundles such as the study of their moduli spaces.

Throughout the article, we work over the field $\mathbb{C}$ of complex numbers. Varieties are smooth and projective. We suppress the variety from the notation for sheaves whenever the underlying variety is clear from the context.

\section{Preliminaries}

\subsection{Fano Varieties}
\begin{definition}
$X$ is called a \emph{Fano variety}, or simply \emph{Fano}, if its anticanonical bundle is ample.
\end{definition}

\begin{example}
The only one-dimensional Fano variety is $\mathbb{P}^1$.
\end{example}

\begin{definition}
A two-dimensional Fano variety $X$ is called a \emph{del Pezzo surface}.
\end{definition}
\begin{remark}
Del Pezzo surfaces are classified as follows: $\mathbb{P}^2$; smooth quadrics $Q \subset \mathbb{P}^3$; a series of surfaces $X_d$ for $d=1, \ldots, 8$ that can be obtained by blowing up $d$ points on $\mathbb{P}^2$ in general position. We refer the reader to \cite{IP1999}*{Examples 2.1.5 (i)} and the references therein for more information.
\end{remark}

The Picard group of $X_d$ is a free abelian group of rank $d+1$. A set of free generators is given by $\ell$, the pullback class of a line in $\mathbb{P}^2$; and the classes $\ell_i$ of the exceptional divisors for $i=1, \ldots, d$. The canonical class is $K_{X_d} = -3 \ell + \sum_{i=1}^d \ell_i$.

For convenience, we state the Riemann-Roch theorem for line bundles on $X_d$.
\begin{theorem}\label{theorem:riemann-roch}
Let $\mathcal{O}(D)$ be a line bundle on $X_d$, where $D = a \ell + \sum_{i=1}^d b_i \ell_i$. Then,
\[
 \chi(\mathcal{O}(D)) = \frac{(a+1)(a+2)}{2} + \sum_{i=1}^d \frac{b_i (1 - b_i)}{2}.
\]
\end{theorem}

Since we shall investigate the trichotomy problem for del Pezzo surfaces with arbitrary polarizations, we also need criteria for a line bundle to be ample. As we only need the \emph{del Pezzo sextic} $X_3$ and the \emph{del Pezzo quintic} $X_4$, the following theorem is stated for those cases.
\begin{theorem}\label{theorem:ampleness}
Let $D = a \ell + b \ell_1 + c \ell_2 + d \ell_3$ be a divisor on $X_3$. Then $\mathcal{O}(D)$ is ample if and only if it is very ample, if and only if $b,c,d \leq -1$, $a+b+c \geq 1$, $a+b+d \geq 1$, $a+c+d \geq 1$.

Let $D = a \ell + b \ell_1 + c \ell_2 + d \ell_3 + e \ell_4$ be a divisor on $X_4$. Then $\mathcal{O}(D)$ is ample if and only if it is very ample, if and only if $b,c,d,e \leq -1$, $a+b+c \geq 1$, $a+b+d \geq 1$, $a+b+e \geq 1$, $a+c+d \geq 1$, $a+c+e \geq 1$, $a+d+e \geq 1$.
\end{theorem}
\begin{proof}
The first equivalences follow from \cite{DiRocco1996}*{Corollary 4.7}. The second equivalences follow from \cite{DiRocco1996}*{Corollary 4.6}.
\end{proof}

\subsection{Ulrich Bundles}
\begin{definition}\label{definition:ulrich}
A vector bundle $\mathcal{E}$ on a polarized variety $(X, \mathcal{O}(H))$ is called an \emph{Ulrich bundle} if $\mbox{h}^{*}(\mathcal{E}(-iH))=0$ for $1 \leq i \leq \mbox{dim}(X)$.
\end{definition}
\begin{remark}\label{remark:ulrich}
There are other, equivalent, definitions of Ulrich bundles; see \cite{Beauville2018}. Note also that if a direct sum of vector bundles is Ulrich, then every summand must be Ulrich. This implies that, whenever we take an Ulrich bundle on a variety, we can take it to be \emph{indecomposable} without loss of generality.
\end{remark}

\begin{lemma}\label{lemma:ulrich_duality}
If $\mathcal{E}$ is an Ulrich bundle on a polarized variety $(X, \mathcal{O}(H))$ of dimension $m$, then $\mathcal{E}^{\vee}((m+1)H+K_X)$ is also an Ulrich bundle on $X$.
\end{lemma}
\begin{proof}
This is an easy consequence of Definition \ref{definition:ulrich}.
\end{proof}

The following definition is adapted from \cite{Miro-Roig2018}*{Definition 3.2.10}; it was stated there for \emph{arithmetically Cohen-Macaulay (ACM)} vector bundles.
\begin{definition}\label{definition:ulrich_representation_type}
Let $(X, \mathcal{O}(H))$ be a projective variety with polarization given by a very ample line bundle.
\begin{enumerate}
 \item $X$ is of \emph{finite Ulrich representation type} if there are only a finite number of indecomposable Ulrich bundles on $X$ up to isomorphism.
 \item $X$ is of \emph{tame Ulrich representation type} if either it has an infinite discrete set of indecomposable Ulrich bundles on $X$ up to isomorphism, or the families of non-isomorphic indecomposable Ulrich bundles of any given rank form a finite number of families of dimension at most $\mbox{dim}(X)$.
 \item $X$ is of \emph{wild Ulrich representation type} if there exist families of non-iso\-mor\-phic indecomposable Ulrich bundles on $X$ with arbitrarily large dimension.
\end{enumerate}
\end{definition}

It was asked in \cite{Miro-Roig2018}*{Problem 3.2.11} whether these three representation types exhausted all the possibilities for ACM vector bundles. In analogy with that problem, we make the following definition.
\begin{definition}\label{definition:ulrich_trichotomy}
Let $(X, \mathcal{O}(H))$ be a projective variety with polarization given by a very ample line bundle. We say that \emph{Ulrich trichotomy holds} for $X$ with the given polarization if $X$ is one of \emph{finite}, \emph{tame}, or \emph{wild Ulrich representation types}.
\end{definition}

\begin{conjecture}\label{conjecture:trichotomy}
Ulrich trichotomy holds for every projective variety $(X, \mathcal{O}(H))$.
\end{conjecture}

In this article, we prove that Conjecture \ref{conjecture:trichotomy} holds for certain varieties.
\begin{theorem}
Ulrich trichotomy holds for $\mathbb{P}^2$, $\mathbb{P}^1 \times \mathbb{P}^1$, the del Pezzo sextic $X_3$, and the del Pezzo quintic $X_4$, with any given polarization.
\end{theorem}
\begin{proof}
See Theorems \ref{theorem:ulrich_trichotomy_for_p2}, \ref{theorem:ulrich_trichotomy_for_p1_times_p1}, \ref{theorem:ulrich_trichotomy_for_x3}, and \ref{theorem:ulrich_trichotomy_for_x4}.
\end{proof}

We finally mention the semistability property of Ulrich bundles, which implies that a (coarse) moduli space of Ulrich bundles exists.
\begin{theorem}
Every Ulrich bundle on $(X, \mathcal{O}(H))$ is Gieseker-semistable.
\end{theorem}
\begin{proof}
See \cite{CKM2012}*{Proposition 2.11}.
\end{proof}
\begin{corollary}\label{corollary:ulrich_h2}
For two Ulrich bundles $\mathcal{E}$ and $\mathcal{F}$, we have $\mbox{h}^2(\mathcal{E} \otimes \mathcal{F}^{\vee})=0$.
\end{corollary}
\begin{proof}
The proof of \cite{CG2017}*{Lemma 2.4} can be suitably modified to yield the conclusion.
\end{proof}
\begin{remark}\label{remark:simple_ulrich_bundle_dimension}
If there exists a \emph{simple} Ulrich bundle $\mathcal{E}$ on $(X, \mathcal{O}(H))$, then $\mathcal{E}$ corresponds to a closed point in a component of the (coarse) moduli space of Ulrich bundles with dimension $\mbox{h}^1(\mathcal{E}^{\vee} \otimes \mathcal{E}) = 1 - \chi(\mathcal{E}^{\vee} \otimes \mathcal{E})$. (We note that $\mbox{h}^2(\mathcal{E}^{\vee} \otimes \mathcal{E})=0$ by Corollary \ref{corollary:ulrich_h2}.)
\end{remark}

\subsection{Quivers and Quiver Representations}
In this subsection, we review the definition and basic properties of quivers. For an introduction to the topic, see \cite{DW2017}.

\begin{definition}
A \emph{quiver} is a directed graph. More formally, a quiver $Q$ is a quadruple $(Q_0, Q_1, s, t)$ where $Q_0$ and $Q_1$ are sets, whose elements are called \emph{vertices} and \emph{arrows} respectively, and two maps $s,t: Q_1 \to Q_0$ that sends each arrow to its \emph{source} and \emph{target} respectively.
\end{definition}
\begin{definition}
A quiver $Q$ is called \emph{finite} if $Q_0$ and $Q_1$ are finite sets. It is called \emph{acyclic} if it does not contain any oriented cycle.
\end{definition}
\begin{remark}
We shall only consider finite and acyclic quivers in this article. With this assumption, the vertices of a quiver $Q$ can be represented as an increasing sequence of integers, where every arrow points in the direction of increasing numbers.
\end{remark}

\begin{definition}
A \emph{representation} $V$ of the quiver $Q$ is a set of finite-dimensional vector spaces $V_i$ for every $i \in Q_0$ and linear maps $V_a: V_i \to V_j$ for every arrow $a: i \to j$. $V$ is called \emph{trivial} if all $V_i = 0$; in this case, we just write $V=0$. Given two representations $V$ and $W$, a \emph{morphism} $\phi: V \to W$ is a collection of linear maps $\phi_i : V_i \to W_i$ such that for every arrow $a: i \to j$, the diagram
\[
 \begin{tikzcd}
  V_i \arrow[r, "V_a"] \arrow["\phi_i", d] & V_j \arrow[d, "\phi_j"] \\
  W_i \arrow[r, "W_a"] & W_j
\end{tikzcd}
\]
commutes.
\end{definition}
\begin{remark}
Since the quiver $Q$ is finite and the $V_i$ were assumed to be finite-dimensional, the dimensions of the $V_i$ can be arranged in a tuple of integers $\underline{d} = \underline{d}(V) \in \mathbb{Z}^{Q_0}$, called the \emph{dimension vector}.
\end{remark}
\begin{remark}\label{remark:rep_q}
The concept of a \emph{direct sum} of quiver representations; \emph{monomorphisms}, \emph{epimorphisms}, \emph{isomorphisms}; and \emph{kernels} and \emph{cokernels} can be defined in a straightforward manner. See \cite{DW2017}*{Chapter 1} for details. As a result, we obtain an abelian category, denoted $\mbox{Rep}(Q)$ of representations of the quiver $Q$. The derived category of $\mbox{Rep}(Q)$ is denoted $\mbox{D}(Q)$.
\end{remark}

Our assumption that $Q$ is acyclic has the following important consequence.
\begin{theorem}
The submodule of every projective module in $\mbox{Rep}(Q)$ is projective. (In other words, $\mbox{Rep}(Q)$ is a \emph{hereditary} category.)
\end{theorem}
\begin{proof}
See \cite{DW2017}*{Theorem 2.3.2}.
\end{proof}
\begin{corollary}
For any two quiver representations $V$ and $W$, $\mbox{Ext}^{i}(V,W)=0$ for $i \geq 2$.
\end{corollary}

This allows us to define the \emph{Euler characteristic} and the \emph{Euler form}.
\begin{definition}
Given two quiver representations $V$ and $W$, the \emph{Euler characteristic} is defined by
\[
 \chi(V, W) = \mbox{hom}(V,W) - \mbox{ext}(V,W).
\]
\end{definition}
\begin{definition}
Given two dimension vectors $\alpha, \beta \in \mathbb{Z}^{Q_0}$, the \emph{Euler form} is defined by
\[
 \langle \alpha, \beta \rangle = \sum_{i \in Q_0} \alpha_i \beta_i - \sum_{a \in Q_1} \alpha_{s(a)} \beta_{t(a)}.
\]
\end{definition}
\begin{remark}
The Euler form is a bilinear form, which is neither symmetric nor skew-symmetric in general.
\end{remark}

The Euler characteristic and the Euler form are related by the following formula.
\begin{proposition}
Given two quiver representations $V$ and $W$, we have
\[
 \chi(V,W) = \langle \underline{d}(V), \underline{d}(W) \rangle.
\]
\end{proposition}
\begin{proof}
See \cite{DW2017}*{Proposition 2.5.2}.
\end{proof}

\begin{example}\label{example:k3_euler_ringel_form}
Consider the \emph{Kronecker quiver} $K_3$ with three arrows
\[
 1 \triplerightarrow 2
\]
and two elements $\underline{d} = (d_1, d_2)$ and $\underline{d}^{\prime} = (d_1^{\prime}, d_2^{\prime})$. Then we have
\[
 \langle \underline{d}, \underline{d}^{\prime} \rangle = d_1 d_1^{\prime} + d_2 d_2^{\prime} - 3 d_1 d_2^{\prime}.
\]
\end{example}

\begin{example}\label{example:s4_euler_ringel_form}
Consider the quiver $S_4$
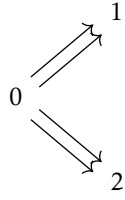
\[
 \begin{tikzcd}
	& 1 & \\
	0\arrow[shift left]{ur}\arrow[shift right,swap]{ur}\arrow[shift left]{dr}\arrow[shift right,swap]{dr} & & \\
	& 2 &
 \end{tikzcd}
\]
and two elements $\underline{d} = (d_0, d_1, d_2)$ and $\underline{d}^{\prime} = (d_0^{\prime}, d_1^{\prime}, d_2^{\prime})$. Then,
\[
 \langle \underline{d}, \underline{d}^{\prime} \rangle = d_0 d_0^{\prime} + d_1 d_1^{\prime} + d_2 d_2^{\prime} - 2 d_0 d_1^{\prime} - 2 d_0 d_2^{\prime}.
\]
\end{example}

\begin{example}\label{example:k32_euler_ringel_form}
Consider the quiver $K_{3,2}$
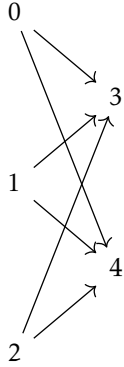
\[
 \begin{tikzcd}
 0
 \arrow[dr]
 \arrow[dddr] & \\
 & 3 \\
 1
 \arrow[ur]
 \arrow[dr] & \\
 & 4 \\
 2
 \arrow[ur]
 \arrow[uuur] &
 \end{tikzcd}
\]
and two elements $\underline{d} = (d_0, d_1, d_2, d_3, d_4)$ and $\underline{d}^{\prime} = (d_0^{\prime}, d_1^{\prime}, d_2^{\prime}, d_3^{\prime}, d_4^{\prime})$. Then,
\[
 \langle \underline{d}, \underline{d}^{\prime} \rangle = \sum_{i=0}^4 d_i d_i^{\prime} - d_0 d_3^{\prime} - d_0 d_4^{\prime} - d_1 d_3^{\prime} - d_1 d_4^{\prime} - d_2 d_3^{\prime} - d_2 d_4^{\prime}.
\]
\end{example}

\begin{example}\label{example:k51_euler_ringel_form}
Consider the quiver $K_{5,1}$
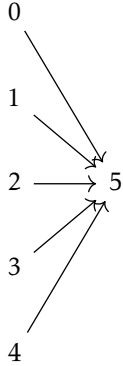
\[
 \begin{tikzcd}
 0\arrow[ddr] \\
 1\arrow[dr] \\
 2\arrow[r] & 5 \\
 3\arrow[ur] \\
 4\arrow[uur]
 \end{tikzcd}
\]
and two elements $\underline{d} = (d_0, d_1, d_2, d_3, d_4, d_5)$ and $\underline{d}^{\prime} = (d_0^{\prime}, d_1^{\prime}, d_2^{\prime}, d_3^{\prime}, d_4^{\prime}, d_5^{\prime})$. Then,
\[
 \langle \underline{d}, \underline{d}^{\prime} \rangle = \sum_{i=0}^5 d_i d_i^{\prime} - \sum_{i=0}^4 d_i d_5^{\prime}.
\]
\end{example}

Given a quiver $Q$, the \emph{path algebra} $\mathbb{C}Q$ is an associative $\mathbb{C}$-algebra defined using \emph{paths}.
\begin{definition}
A \emph{path} $p$ in a quiver $Q$ of length $\ell$ is a sequence of $\ell$ arrows where the head of each arrow is the tail of the next arrow. The \emph{tail} $tp$ of $p$ is the first vertex in the sequence; the \emph{head} $hp$ of $p$ is the last vertex in the sequence. For every $i \in Q_0$, we have the \emph{trivial path} $e_i$ with head and tail both equal to $i$ and the length of $e_i$ is 0.
\end{definition}
\begin{definition}
The \emph{path algebra} $\mathbb{C}Q$ is the associative $\mathbb{C}$-algebra with basis given by the set of all paths in $Q$ and with multiplication $pq$ of two paths $p$ and $q$ given by the concatenation of paths if $tp=hq$ and 0 otherwise.
\end{definition}
\begin{remark}
It can be easily seen that there are finitely many paths in an \emph{acyclic} quiver $Q$. Hence, in this case, $\mathbb{C}Q$ is a finite-dimensional algebra.
\end{remark}

Let $\mathfrak{m} \subset \mathbb{C}Q$ be the two-sided ideal generated by all the arrows in $Q$.
\begin{definition}
A two-sided ideal $J \subset \mathbb{C}Q$ such that $\mathfrak{m}^k \subset J \subset \mathfrak{m}^2$ for some integer $k \geq 2$ is called \emph{admissible}.
\end{definition}

Admissible ideals are used to define quivers with relations.
\begin{definition}
A \emph{quiver with relations} $(Q;J)$ is a quiver $Q$ together with an admissible ideal $J \subset \mathbb{C}Q$. A \emph{representation} of a quiver with relations $(Q;J)$ is a representation $V$ of $Q$ that satisfies all relations $r \in J$.
\end{definition}
\begin{remark}
Just as in Remark \ref{remark:rep_q}, the representations of a quiver with representations $(Q;J)$ form an abelian category, denoted $\mbox{Rep}(Q;J)$. The derived category of $\mbox{Rep}(Q;J)$ is denoted $\mbox{D}(Q;J)$.
\end{remark}

\begin{remark}\label{remark:hyperbolic_quivers}
Quivers are classified as \emph{finite}, \emph{affine}, and \emph{indefinite} types. Quivers of finite and affine types can further be classified as members of finitely many families. See \cite{Kac1990}*{Chapter 4} for details and lists of families of quivers of finite and affine types.

A quiver $Q$ of indefinite type is called \emph{hyperbolic} if any connected proper subdiagram of $Q$ is of finite or affine type. It follows that the quivers $K_3$, $S_4$, $K_{3,2}$, and $K_{5,1}$ are hyperbolic.
\end{remark}

\subsection{Exceptional Collections}
In this subsection, $\mbox{D}(X)$ denotes the bounded derived category of coherent sheaves on $X$. The definitions and notations in this subsection follow those in \cite{Maiorana2017}. See also \cite{GK2004}.

\begin{definition}
An object $E \in \mbox{D}(X)$ is called \emph{exceptional} if
\[
 \mbox{Hom}(E, E[k]) =
 \begin{cases}
  \mathbb{C} & \text{if $k = 0$,} \\
  0 & \text{if $k \neq 0$.}
 \end{cases}
\]
\end{definition}


\begin{definition}
A sequence $\mathfrak{E} = (E_0, \ldots, E_n)$ of exceptional objects in $\mbox{D}(X)$ is called an \emph{exceptional sequence} if $\mbox{Hom}(E_i,E_j[k])=0$ for all $i > j$ and all $k$. It is called a \emph{strong exceptional sequence} if, in addition, $\mbox{Hom}(E_i,E_j[k])=0$ for all $i,j$ and all $k \neq 0$. It is called \emph{full} if the objects $E_0, \ldots, E_n$ generate $\mbox{D}(X)$.
\end{definition}

\begin{definition}\label{definition:exceptional_collections}
Given an exceptional collection $\mathfrak{E} = (E_0, \ldots, E_n)$, two exceptional collections $^{\vee}\mathfrak{E} = (^{\vee}E_n, \ldots, ^{\vee}E_0)$ and $\mathfrak{E}^{\vee} = (E_n^{\vee}, \ldots, E_0^{\vee})$ are called \emph{left dual} and \emph{right dual} collections, respectively, if
\[
 \mbox{Hom}(^{\vee}E_i, E_j[k]) =
 \begin{cases}
  \mathbb{C} & \text{if $i = j = n-k$,} \\
  0 & \text{otherwise,}
 \end{cases}
\]
and
\[
 \mbox{Hom}(E_i, E_j^{\vee}[k]) =
 \begin{cases}
  \mathbb{C} & \text{if $i=j=k$,} \\
  0 & \text{otherwise.}
 \end{cases}
\]
\end{definition}

Given a full exceptional collection $\mathfrak{E} = (E_0, \ldots, E_n)$, left dual and right dual collections to $\mathfrak{E}$ exist and are unique. They can be obtained by mutations.
\begin{definition}[cf. \cite{GK2004}*{2.2}]
Given an exceptional object $E \in \mbox{D}(X)$, and any object $F \in \mbox{D}(X)$, the \emph{left mutation} $\mathbb{L}_E (F)$ and the \emph{right mutation} $\mathbb{R}_E (F)$ of $F$ with respect to $E$ are defined by the distinguished triangles
\[
 \mathbb{L}_E (F) \to \mbox{RHom}(E,F) \otimes E \to F \to \mathbb{L}_E (F)[1]
\]
and
\[
 \mathbb{R}_E (F)[-1] \to F \to \mbox{RHom}(F,E)^{\vee} \otimes E \to \mathbb{R}_E (F).
\]
\end{definition}
\begin{remark}[cf. \cite{GK2004}*{2.6}]
Given a full exceptional collection $\mathfrak{E} = (E_0, \ldots, E_n)$, its left dual and right dual collections can be constructed using consecutive mutations by
\[
 ^{\vee}E_{n-k} = \mathbb{R}_{E_n} \mathbb{R}_{E_{n-1}} \ldots \mathbb{R}_{E_{n-k+1}} E_{n-k}
\]
and
\[
 E_{n-k}^{\vee} = \mathbb{L}_{E_0} \mathbb{L}_{E_1} \ldots \mathbb{L}_{E_{n-k-1}} E_{n-k}.
\]
(We note that our indexing differs from that of \cite{GK2004}.)
\end{remark}
\begin{remark}\label{remark:twisting_collections}
Given a full exceptional collection $\mathfrak{E}$, its left and right duals $^{\vee}E$ and $E^{\vee}$, and a line bundle $\mathcal{L}$ on $X$, the elements in these collections can all be twisted by $\mathcal{L}$ and the duality relations will be preserved. Strongness is also preserved by twisting by $\mathcal{L}$.
\end{remark}

\begin{example}[cf. \cite{Maiorana2017}*{Example 2.29 (1)}]\label{example:p2_exceptional_collections}
For any $d \in \mathbb{Z}$, the derived category $\mbox{D}(\mathbb{P}^2)$ admits a full strong exceptional collection
\[
 \mathfrak{E} = (E_0, E_1, E_2) = (\mathcal{O}(d-3), \mathcal{O}(d-2), \mathcal{O}(d-1)),
\]
which admits the left dual
\[
 ^{\vee}\mathfrak{E} = (^{\vee}E_2, ^{\vee}E_1, ^{\vee}E_0) = (\mathcal{O}(d-1), \Omega(d+1), \mathcal{O}(d)),
\]
which is also strong.
\end{example}

\begin{example}[cf. \cite{Maiorana2017}*{Example 2.29 (2)}]\label{example:p1_times_p1_exceptional_collections}
For any $a, b \in \mathbb{Z}$, the derived category $\mbox{D}(\mathbb{P}^1 \times \mathbb{P}^1)$ admits a full exceptional collection
\[
 \mathfrak{E} = (E_0, E_1, E_2, E_3) = (\mathcal{O}(a-1,b-1)[-1], \mathcal{O}(a-1,b)[-1], \mathcal{O}(a,b-1), \mathcal{O}(a,b)),
\]
which admits the left dual
\[
 ^{\vee}\mathfrak{E} = (^{\vee}E_3, ^{\vee}E_2, ^{\vee}E_1, ^{\vee}E_0) = (\mathcal{O}(a,b), \mathcal{O}(a,b+1), \mathcal{O}(a+1,b), \mathcal{O}(a+1,b+1)).
\]
$^{\vee}\mathfrak{E}$ is a strong exceptional collection, but $\mathfrak{E}$ is not.
\end{example}

\begin{example}\label{example:x3_exceptional_collections}
Let $X_3$ denote the \emph{del Pezzo sextic}, i.e. the del Pezzo surface obtained by blowing up three points in general position in $\mathbb{P}^2$. Let $\pi: X_3 \to \mathbb{P}^2$ be the blowdown map. The Picard group of $X_3$ is a free abelian group of rank four; the generators are $\ell$, the pullback of the class of a line in $\mathbb{P}^2$, and the three exceptional divisors $\ell_i$ for $i=1,2,3$.

By \cite{KN1998}*{Proposition 4.2}, there exists a full strong exceptional collection on $X_3$, given by
\[
 (F_0, F_1, F_2, F_3, F_4, F_5) =
\]
\[
 (\mathcal{O}, \mathcal{O}(2\ell-\ell_1-\ell_2-\ell_3), \mathcal{O}(\ell), \mathcal{O}(2\ell-\ell_1-\ell_2),\mathcal{O}(2\ell-\ell_1-\ell_3),\mathcal{O}(2\ell-\ell_2-\ell_3)).
\]
Using left mutations, this can be seen to have the right dual collection
\[
 (F_5^{\vee}, F_4^{\vee}, F_3^{\vee}, F_2^{\vee}, F_1^{\vee}, F_0^{\vee}) =
\]
\[
 (\mathcal{O}(\ell_1-\ell)[-3], \mathcal{O}(\ell_2-\ell)[-2], \mathcal{O}(\ell_3-\ell)[-1], \mathcal{P}[-1], \mathcal{K}, \mathcal{O}),
\]
where $\mathcal{P}$ is given by
\[
 0 \to \mathcal{P} \to \mathcal{O}^3 \to \mathcal{O}(\ell) \to 0,
\]
and $\mathcal{K}$ is given by
\[
 0 \to \mathcal{K} \to \mathcal{O}^3 \to \mathcal{O}(2 \ell - \ell_1 - \ell_2 - \ell_3) \to 0.
\]
By Definition \ref{definition:exceptional_collections}, it is clear that, if we now relabel
\[
 \mathfrak{E} = (E_0, E_1, E_2, E_3, E_4, E_5) =
\]
\[
 (\mathcal{O}(\ell_1-\ell)[-3], \mathcal{O}(\ell_2-\ell)[-2], \mathcal{O}(\ell_3-\ell)[-1], \mathcal{P}[-1], \mathcal{K}, \mathcal{O}),
\]
then the left dual collection is given by
\[
 ^{\vee}\mathfrak{E} = (^{\vee}E_5, ^{\vee}E_4, ^{\vee}E_3, ^{\vee}E_2, ^{\vee}E_1, ^{\vee}E_0) =
\]
\[
 (\mathcal{O}, \mathcal{O}(2\ell-\ell_1-\ell_2-\ell_3), \mathcal{O}(\ell), \mathcal{O}(2\ell-\ell_1-\ell_2),\mathcal{O}(2\ell-\ell_1-\ell_3),\mathcal{O}(2\ell-\ell_2-\ell_3)).
\]
\end{example}

\begin{example}\label{example:x4_exceptional_collections}
Let $X_4$ denote the \emph{del Pezzo quintic}, i.e. the del Pezzo surface obtained by blowing up four points in general position in $\mathbb{P}^2$. Let $\pi: X_4 \to \mathbb{P}^2$ be the blowdown map. The Picard group of $X_4$ is a free abelian group of rank five; the generators are $\ell$, the pullback of the class of a line in $\mathbb{P}^2$, and the four exceptional divisors $\ell_i$ for $i=1,2,3,4$.

By \cite{KN1998}*{Proposition 4.2}, there exists a full strong exceptional collection on $X_4$, given by
\begin{align*}
 (F_0, F_1, F_2, F_3, F_4, F_5, F_6) = (& \mathcal{O}, \mathcal{F}, \mathcal{O}(\ell), \mathcal{O}(2\ell - \ell_2 - \ell_3 - \ell_4), \mathcal{O}(2\ell - \ell_1 - \ell_3 - \ell_4), \\
 & \mathcal{O}(2\ell - \ell_1 - \ell_2 - \ell_4), \mathcal{O}(2\ell - \ell_1 - \ell_2 - \ell_3).
\end{align*}
Here, $\mathcal{F}$ is a rank-2 vector bundle obtained as the universal extension
\[
 0 \to \mathcal{O}(2\ell - \ell_1 - \ell_2 - \ell_3 - \ell_4) \to \mathcal{F} \to \mathcal{O}(\ell) \to 0.
\]
Using left mutations, this can be seen to have the right dual collection
\[
 (F_6^{\vee}, F_5^{\vee}, F_4^{\vee}, F_3^{\vee}, F_2^{\vee}, F_1^{\vee}, F_0^{\vee}) =
\]
\[
 (\mathcal{O}(\ell_4-\ell)[-4], \mathcal{O}(\ell_3-\ell)[-3], \mathcal{O}(\ell_2-\ell)[-2], \mathcal{O}(\ell_1-\ell)[-1], \mathcal{O}(-2\ell + \ell_1 + \ell_2 + \ell_3 + \ell_4), \mathcal{K}, \mathcal{O}),
\]
where $\mathcal{K}$ is given by
\[
 0 \to \mathcal{K} \to \mathcal{O}^5 \to \mathcal{F} \to 0.
\]
By Definition \ref{definition:exceptional_collections}, it is clear that, if we now relabel
\[
 \mathfrak{E} = (E_0, E_1, E_2, E_3, E_4, E_5, E_6) =
\]
\[
 (\mathcal{O}(\ell_4-\ell)[-4], \mathcal{O}(\ell_3-\ell)[-3], \mathcal{O}(\ell_2-\ell)[-2], \mathcal{O}(\ell_1-\ell)[-1], \mathcal{O}(-2\ell + \ell_1 + \ell_2 + \ell_3 + \ell_4), \mathcal{K}, \mathcal{O}),
\]
then the left dual collection is given by
\begin{align*}
 ^{\vee}\mathfrak{E} &= (^{\vee}E_6, ^{\vee}E_5, ^{\vee}E_4, ^{\vee}E_3, ^{\vee}E_2, ^{\vee}E_1, ^{\vee}E_0) = (\mathcal{O}, \mathcal{F}, \mathcal{O}(\ell), \mathcal{O}(2\ell - \ell_2 - \ell_3 - \ell_4), \\
 & \mathcal{O}(2\ell - \ell_1 - \ell_3 - \ell_4), \mathcal{O}(2\ell - \ell_1 - \ell_2 - \ell_4), \mathcal{O}(2\ell - \ell_1 - \ell_2 - \ell_3).
\end{align*}
\end{example}

\begin{theorem}\label{theorem:bondal}
Suppose there exists a full strong exceptional collection $(^{\vee}E_n, \ldots, ^{\vee}E_0)$ on $D(X)$. Then there exists a quiver with relations $(Q;J)$, and a triangulated equivalence
\[
 \Phi: \mbox{D}(X) \to \mbox{D}(Q; J).
\]
\end{theorem}
\begin{proof}
Let $T := \oplus_{i=0}^{n} {^{\vee}E_i} \in \mbox{D}(X)$. Since the endomorphism algebra $A = End(T)$ is basic, it is isomorphic to $(\mathbb{C}Q/J)^{op}$ for a quiver with relations $(Q;J)$. The vertices of this quiver can be labeled with the integers $\{ 0, \ldots, n \}$ and every arrow points in the direction of increasing numbers. Then, the functor
\[
 \Phi = R\mbox{Hom}(T, \cdot) : \mbox{D}(X) \to \mbox{D}(Q;J),
\]
which maps an object $V \in \mbox{D}(X)$ to a complex of representations that has the graded vector space $R\mbox{Hom}(^{\vee}E_i,V)$ at the vertex labeled $i$, is a triangulated equivalence. See \cite{Bondal1989}*{Theorem 6.2} for the details of the proof; the version given here is adapted from \cite{Maiorana2017}*{Theorem 2.30}.
\end{proof}
\begin{remark}[cf. \cite{Maiorana2017}*{Remark 2.31}]\label{remark:phi}
$\Phi$ maps $^{\vee}E_i$ to the projective representation $P_i$ of $Q$ corresponding to vertex $i$, and $E_i$ to $S_i [i-n]$, where $S_i$ is the simple representation of $Q$ corresponding to vertex $i$. It follows that the standard heart $\mbox{Rep}(Q;J) \subset \mbox{D}(Q;J)$ corresponds, via $\Phi$ to the extension closure of $E_i [n-i]$ for $i=0, \ldots, n$.
\end{remark}
\begin{remark}
In \cite{Maiorana2017}*{Section 5}, the author includes a shift of 1 in the equivalence $\Phi$. Since this shift is useless for our purposes, we do not include it in this article.
\end{remark}

\section{The Trichotomy Problem}

\subsection{$\mathbb{P}^2$ as a special case}
The case of the projective plane $\mathbb{P}^2$ provides a very simple illustration of the technique of using the relationships between Ulrich bundles and quiver representations to give a solution to the trichotomy problem. Ulrich bundles on $\mathbb{P}^2$, along with projective spaces of higher dimensions, have been investigated in \cite{ESW2003}, and more recently in \cite{CG2017} and \cite{Lin2017}.

Throughout this subsection, all sheaves and cohomologies are considered on $\mathbb{P}^2$ without explicit mention. Also, Ulrich bundles are considered with respect to an arbitrary but fixed polarization $\mathcal{O}(d)$, where $d \geq 1$.

We restate Definition \ref{definition:ulrich} for vector bundles on $\mathbb{P}^2$.
\begin{definition}\label{definition:ulrich_p2}
An \emph{Ulrich bundle} is a vector bundle $\mathcal{E}$ such that $h^{*}(\mathcal{E}(-id)) = 0$ for $i=1,2$.
\end{definition}

Now let $\mathcal{E}$ be an Ulrich bundle of rank $r$. As described in \cite{Maiorana2017}*{Section 5.1}, the full strong collections of Example \ref{example:p2_exceptional_collections} give rise to a tilting sheaf $T := \oplus_{i=0}^{2} {^{\vee}E_i}$ whose endomorphism algebra $A = End(T)$ is isomorphic to the opposite of the bound quiver algebra $\mathbb{C}B_3 / J$ of the \emph{Beilinson quiver}
\[
 B_3:
 \begin{tikzcd}
  0\arrow[bend left=50]{r}{a_1}\arrow{r}{a_2}\arrow[bend right=50]{r}{a_3} &1\arrow[bend left=50]{r}{b_1}\arrow{r}{b_2}\arrow[bend right=50]{r}{b_3} &2
 \end{tikzcd}
\]
where the relations are given by the two-sided ideal $J = (b_i a_j + b_j a_i\, |\, i,j=1,2,3) \subset \mathbb{C}B_3$. Under the triangulated equivalence $\Phi$ described in Theorem \ref{theorem:bondal}, the Ulrich bundle $\mathcal{E}$ is mapped to the complex of quiver representations

\[
 R\mbox{Hom}(\mathcal{O}(d), \mathcal{E}) \triplerightarrow R\mbox{Hom}(\Omega(d+1), \mathcal{E}) \triplerightarrow R\mbox{Hom}(\mathcal{O}(d-1), \mathcal{E}).
\]

The first term vanishes by Definition \ref{definition:ulrich_p2}. The second and third terms are concentrated in degree 0 by \cite{CG2017}*{Theorem 4.2}. It follows that we can consider $\Phi(\mathcal{E})$ to be a complex of representations of the Kronecker quiver $K_3$ concentrated in degree 0, or equivalently, a representation of $K_3$.

\begin{theorem}\label{theorem:ulrich_trichotomy_for_p2}
Ulrich trichotomy holds for $(\mathbb{P}^2, \mathcal{O}(dH))$.
\end{theorem}
\begin{proof}
If there are no Ulrich bundles, $(\mathbb{P}^2, \mathcal{O}(d))$ is of finite Ulrich representation type.

Now let us assume that there are Ulrich bundles. Recall that the quiver $K_3$ is hyperbolic by Remark \ref{remark:hyperbolic_quivers}.

\emph{Case 1.} There exists an Ulrich bundle $\mathcal{E}$ of rank $r$, which we can assume to be indecomposable without loss of generality by Remark \ref{remark:ulrich}, such that the dimension vector $\underline{d} = (a,b)$ of the associated representation of the quiver $K_3$ satisfies $\chi(\underline{d}, \underline{d}) < 0$. By \cite{Kac1980}*{Proposition 1.6 (a)}, $\underline{d}$ is an imaginary root. Then, by \cite{Kac1980}*{Lemma 2.1 (e) and Lemma 2.7}, $\underline{d}$ is a Schur root. By \cite{Schofield1992}*{Theorem 3.7}, $n \underline{d}$ is a Schur root for every $n \geq 1$. Hence, in $Rep(n \underline{d})$, the general element is simple. Since $\mathcal{E}^{\oplus n}$ is an Ulrich bundle, $\Phi(\mathcal{E}^{\oplus n})$ has dimension vector $n \underline{d}$, and the condition of being Ulrich is open, the general element in $Rep(n \underline{d})$ corresponds via $\Phi$ to a simple, hence indecomposable, Ulrich bundle.

The dimension of the moduli space of these Ulrich bundles can be computed using Remark \ref{remark:simple_ulrich_bundle_dimension} and the Euler-Ringel form as
\[
 1 - \chi(n \underline{d}, n \underline{d}) = 1 - n^2 \chi(\underline{d}, \underline{d}).
\]
It is clear that this dimension tends to infinity as $n \to \infty$. Hence, $(\mathbb{P}^2, \mathcal{O}(d))$ is of wild Ulrich representation type.

\emph{Case 2.} There exists no Ulrich bundle with dimension vector $\underline{d} = (a,b)$ of the associated representation of the quiver $K_3$ such that $\chi(\underline{d}, \underline{d}) < 0$, but there exists an indecomposable Ulrich bundle $\mathcal{E}$ such that the associated representation of the quiver $K_3$ has dimension vector $\underline{d} = (a,b)$ that satisfies $\chi(\underline{d}, \underline{d}) = 0$. Since $\chi(\underline{d}, \underline{d}) = a^2 + b^2 - 3ab$ by Example \ref{example:k3_euler_ringel_form}, and it is clear that this can never be equal to 0 for $a,b \in \mathbb{Z}$, this case can be ruled out.

\emph{Case 3.} There exists no Ulrich bundle with dimension vector $\underline{d} = (a,b)$ of the associated representation of the quiver $K_3$ such that $\chi(\underline{d}, \underline{d}) \leq 0$, but there exists an Ulrich bundle $\mathcal{E}$ such that the associated representation of the quiver $K_3$ has dimension vector $\underline{d} = (a,b)$ that satisfies $\chi(\underline{d}, \underline{d}) > 0$. Note that $\mathcal{E}$ can again be assumed to be indecomposable without loss of generality. By \cite{Kac1980}*{Lemma 1.9 (b)}, $\underline{d}$ is a real root, and by \cite{Kac1980}*{Proposition 1.6 (b)}, $\chi(\underline{d}, \underline{d}) = 1$. Finally, by \cite{Kac1980}*{Lemma 2.1 (e) and Theorem 2 (c)}, $\mathcal{E}$ is the unique indecomposable Ulrich bundle with dimension vector equal to $\underline{d}$ of the associated representation of the quiver $K_3$. The proof will now be complete if we show that there are only finitely many such $\underline{d}$, since this implies that $(\mathbb{P}^2, \mathcal{O}(d))$ is of finite Ulrich representation type.

By Remark \ref{remark:phi}, the Ulrich bundle $\mathcal{E}$ fits into a short exact sequence of the form
\[
 0 \to \mathcal{O}(d-2)^{a} \to \mathcal{O}(d-1)^{b} \to \mathcal{E} \to 0.
\]

Since $\mathcal{E}$ is an Ulrich bundle, we must have $\chi(\mathcal{E}(-2d))=0$ by Definition \ref{definition:ulrich_p2}. Twisting the above short exact sequence by $\mathcal{O}(-2d)$ and calculating Euler characteristics, we have $a \chi(\mathcal{O}(-d-2)) = b \chi(\mathcal{O}(-d-1))$. Simplifying, this gives
\[
 \frac{a}{b} = \frac{d-1}{d+1}.
\]
The equation $\chi(\underline{d}, \underline{d}) = 1$ gives
\[
 a^2 + b^2 - 3ab = 1
\]
by Example \ref{example:k3_euler_ringel_form}. It is now clear that these two equations have finitely many solutions; hence there are only finitely many possible $\underline{d}$.
\end{proof}

\begin{remark}
By \cite{CG2017}*{Proposition 3.1} and \cite{CG2017}*{Theorem 5.4}, $(\mathbb{P}^2, \mathcal{O}(1))$ and $(\mathbb{P}^2, \mathcal{O}(2))$ are of finite Ulrich representation type. For even $d \geq 3$, $(\mathbb{P}^2, \mathcal{O}(d))$ is of wild Ulrich representation type by \cite{CG2017}*{Theorem 6.1}. For odd $d \geq 3$, the same result follows by \cite{CG2017}*{Theorem 6.2} and \cite{CMR2018}*{Theorem 1}.
\end{remark}

\subsection{$\mathbb{P}^1 \times \mathbb{P}^1$}
A slightly more sophisticated example concerns the trichotomy problem for Ulrich bundles on $\mathbb{P}^1 \times \mathbb{P}^1$. Ulrich bundles on ruled surfaces on $\mathbb{P}^1$ have been recently investigated by Antonelli in \cite{Antonelli2018}.

Throughout this subsection, all sheaves and cohomologies are considered on $\mathbb{P}^1 \times \mathbb{P}^1$ without explicit mention. Also, Ulrich bundles are considered with respect to an arbitrary but fixed polarization $\mathcal{O}(H) := \mathcal{O}(a, b)$, where $a, b \geq 1$.

We restate Definition \ref{definition:ulrich} for vector bundles on $\mathbb{P}^1 \times \mathbb{P}^1$.
\begin{definition}\label{definition:ulrich_p1_times_p1}
An \emph{Ulrich bundle} is a vector bundle $\mathcal{E}$ such that $h^{*}(\mathcal{E}(-iH)) = 0$ for $i=1,2$.
\end{definition}

In this case again, as described in \cite{Maiorana2017}*{Section 6}, the tilting sheaf $T := \oplus_{i=0}^{3} {^{\vee}E_i}$ has an endormorphism algebra $A$ which is isomorphic to the opposite of the bound quiver algebra $\mathbb{C}Q_4 / J$ of the quiver
\[
 Q_4:
  \begin{tikzcd}
	&1\arrow[shift left]{dr}{b^1_1}\arrow[shift right,swap]{dr}{b^1_2} &\\
	0\arrow[shift left]{ur}{a^1_1}\arrow[shift right,swap]{ur}{a^1_2}\arrow[shift left]{dr}{a^2_1}\arrow[shift right,swap]{dr}{a^2_2} &&3\\
	&2\arrow[shift left]{ur}{b^2_1}\arrow[shift right,swap]{ur}{b^2_2}&
  \end{tikzcd}
\]
with relations given by the two-sided ideal $J = (b_i^1 a_j^1 + b_j^2 a_i^2\, |\, i,j=1,2) \subset \mathbb{C}Q_4$. Let $\mathcal{E}$ be an Ulrich bundle of rank $r$. Under the triangulated equivalence $\Phi$ of Theorem \ref{theorem:bondal}, $\mathcal{E}$ is mapped to the complex of quiver representations
\[
 \begin{tikzcd}[column sep=tiny]
	&&R\mbox{Hom}(\mathcal{O}(a+1,b),\mathcal{E})\arrow[shift left]{dr}\arrow[shift right,swap]{dr} &\\
	&R\mbox{Hom}(\mathcal{O}(a+1,b+1),\mathcal{E})\arrow[shift left]{ur}\arrow[shift right,swap]{ur}\arrow[shift left]{dr}\arrow[shift right,swap]{dr}&&R\mbox{Hom}(\mathcal{O}(a,b),\mathcal{E})\\
	&&R\mbox{Hom}(\mathcal{O}(a,b+1),\mathcal{E})\arrow[shift left]{ur}\arrow[shift right,swap]{ur}&
 \end{tikzcd}.
\]
The term on the right vanishes by Definition \ref{definition:ulrich_p1_times_p1}. The remaining terms are all concentrated in the first cohomologies by the calculations in the proof of \cite{Antonelli2018}*{Theorem 3.2}. We denote their dimensions by $\gamma$ for the left term, $\delta$ for the top term, and $\tau$ for the bottom term. It follows that we can consider $\Phi(\mathcal{E})$ to be a complex of representations of the bound quiver $(Q_4;J)$ concentrated in degree 1, or equivalently, a representation of the quiver $S_4$ with dimension vector $\underline{d} = (\gamma, \delta, \tau)$.

The first main result of this article is the following theorem.
\begin{theorem}\label{theorem:ulrich_trichotomy_for_p1_times_p1}
Ulrich trichotomy holds for $(\mathbb{P}^1 \times \mathbb{P}^1, \mathcal{O}(H))$.
\end{theorem}
\begin{proof}
If there are no Ulrich bundles, $(\mathbb{P}^1 \times \mathbb{P}^1, \mathcal{O}(H))$ is of finite Ulrich representation type.

Now let us assume that there are Ulrich bundles. Recall that the quiver $S_4$ is hyperbolic by Remark \ref{remark:hyperbolic_quivers}.

\emph{Case 1.} There exists an Ulrich bundle $\mathcal{E}$, which we can assume to be indecomposable without loss of generality by Remark \ref{remark:ulrich}, such that the dimension vector $\underline{d} = (\gamma, \delta, \tau)$ of the associated representation of the quiver $S_4$ satisfies $\chi(\underline{d}, \underline{d}) < 0$. In this case, $(\mathbb{P}^1 \times \mathbb{P}^1, \mathcal{O}(H))$ is of wild Ulrich representation type. The argument follows that of Theorem \ref{theorem:ulrich_trichotomy_for_p2}, so we omit the details.

\emph{Case 2.} There exists no Ulrich bundle with dimension vector $\underline{d} = (\gamma, \delta, \tau)$ of the associated representation of the quiver $S_4$ such that $\chi(\underline{d}, \underline{d}) < 0$, but there exists an indecomposable Ulrich bundle $\mathcal{E}$ such that the associated representation of the quiver $S_4$ has dimension vector $\underline{d} = (\gamma, \delta, \tau)$ that satisfies $\chi(\underline{d}, \underline{d}) = 0$. Just as in the proof of Case 1 of Theorem \ref{theorem:ulrich_trichotomy_for_p2}, we have that by \cite{Kac1980}*{Proposition 1.6 (a)}, $\underline{d}$ is an imaginary root, and then by \cite{Kac1980}*{Lemma 2.1 (e) and Lemma 2.7}, $\underline{d}$ is a Schur root. We note that the dimension of the moduli space of these Ulrich bundles is given by
\[
 1 - \chi(\underline{d}, \underline{d}) = 1.
\]
We must now show that for a given rank $r$, there are finitely many such $\underline{d}$, and hence that $(\mathbb{P}^1 \times \mathbb{P}^1, \mathcal{O}(H))$ is of tame Ulrich representation type. Since the details are identical to the proof of Case 3 below, we omit them here.

\emph{Case 3.} There exists no Ulrich bundle with dimension vector $\underline{d} = (\gamma, \delta, \tau)$ of the associated representation of the quiver $S_4$ such that $\chi(\underline{d}, \underline{d}) \leq 0$, but there exists an Ulrich bundle $\mathcal{E}$ of rank $r$ such that the associated representation of the quiver $S_4$ has dimension vector $\underline{d} = (\gamma, \delta, \tau)$ that satisfies $\chi(\underline{d}, \underline{d}) > 0$. Note that $\mathcal{E}$ can again be assumed to be indecomposable without loss of generality. Just as in the proof of Theorem \ref{theorem:ulrich_trichotomy_for_p2}, by \cite{Kac1980}*{Proposition 1.9 (b)}, $\underline{d}$ is a real root, and by \cite{Kac1980}*{Proposition 1.6 (b)}, $\chi(\underline{d}, \underline{d}) = 1$. Again, by \cite{Kac1980}*{Lemma 2.1 (e) and Theorem 2 (c)}, $\mathcal{E}$ is the unique indecomposable Ulrich bundle with dimension vector equal to $\underline{d}$ of the associated representation of the quiver $S_4$; and the proof will be complete if we show that there are only finitely many such $\underline{d}$ for the given rank $r$.

By Remark \ref{remark:phi}, the Ulrich bundle $\mathcal{E}$ fits into a short exact sequence of the form
\[
 0 \to \mathcal{O}(a-1,b-1)^{\gamma} \to \mathcal{O}(a-1,b)^{\delta} \oplus \mathcal{O}(a,b-1)^{\tau} \to \mathcal{E} \to 0.
\]
It follows immediately that
\[
 \delta + \tau - \gamma = r.
\]

Next, by Definition \ref{definition:ulrich_p1_times_p1}, we must have $\chi(\mathcal{E}(-2H))=0$. Twisting the short exact sequence above by $\mathcal{O}(-2H)$ and calculating Euler characteristics, and simplifying, we have
\[
 \gamma ab = \delta a (b-1) + \tau b(a-1).
\]
Finally, the equation $\chi(\underline{d}, \underline{d}) = 1$ gives
\[
 \gamma^2 + \delta^2 + \tau^2 - 2\gamma \delta -2\gamma \tau = 1
\]
by Example \ref{example:s4_euler_ringel_form}. Now, if these three equations have finitely many solutions, then $(\mathbb{P}^1 \times \mathbb{P}^1, \mathcal{O}(a, b))$ is of either finite Ulrich representation type or tame Ulrich representation type.

Suppose then that these three equations have infinitely many solutions. We may consider these solutions to be in $\mathbb{R}^3$, where they form a discrete set, and hence in real projective space $\mathbb{R} \mathbb{P}^3$, with the extra coordinate given by $z$. Since these solutions form an infinite set in $\mathbb{R} \mathbb{P}^3$, and $\mathbb{R} \mathbb{P}^3$ is a compact topological space, the set of solutions must have an accumulation point, which will then necessarily satisfy $z=0$. Consider such an accumulation point $[\gamma_0 : \delta_0 : \tau_0 : 0]$. It is then clear that $\gamma_0$, $\delta_0$, and $\tau_0$ satisfy the equations
\begin{equation}\label{equation:p1_times_p1_4}
 \delta_0 + \tau_0 - \gamma_0 = 0,
\end{equation}
\begin{equation}\label{equation:p1_times_p1_3}
 \gamma_0 ab = \delta_0 a (b-1) + \tau_0 b(a-1),
\end{equation}
\begin{equation}\label{equation:p1_times_p1_1}
 \gamma_0^2 + \delta_0^2 + \tau_0^2 - 2\gamma_0 \delta_0 -2\gamma_0 \tau_0 = 0.
\end{equation}
\ref{equation:p1_times_p1_4} then gives us
\begin{equation}\label{equation:p1_times_p1_2}
 \gamma_0 = \delta_0 + \tau_0,
\end{equation}
and \ref{equation:p1_times_p1_3} gives us
\begin{equation}
 \gamma_0 = \delta_0 \frac{b-1}{b} + \tau_0 \frac{a-1}{a}.
\end{equation}

Plugging \ref{equation:p1_times_p1_2} into \ref{equation:p1_times_p1_1} gives
\begin{equation}
 (\delta_0 + \tau_0)^2 + \delta_0^2 + \tau_0^2 -2(\delta_0 + \tau_0)^2 = 0.
\end{equation}
Simplifying, we get $\delta_0 \tau_0 = 0$. If we now take $\delta_0 = 0$ without loss of generality, \ref{equation:p1_times_p1_2} gives $\gamma_0 = \tau_0$; and \ref{equation:p1_times_p1_3} gives 
\begin{equation}
 \gamma_0 = \frac{a-1}{a} \tau_0,
\end{equation}
which forces $\gamma_0 = \tau_0 = 0$, an impossibility.
\end{proof}

\subsection{The Del Pezzo Sextic $X_3$}

We come to the second main result of this article, namely that Conjecture \ref{conjecture:trichotomy} holds for $X_3$.

Throughout this subsection, all sheaves and cohomologies are considered on $X_3$ without explicit mention. Also, Ulrich bundles are considered with respect to an arbitrary but fixed polarization $\mathcal{O}(H)$ where $H = a \ell + b \ell_1 + c \ell_2 + d \ell_3$. We note that $\mathcal{O}(H)$ is both ample and very ample by Theorem \ref{theorem:ampleness}.

We restate Definition \ref{definition:ulrich} for vector bundles on $X_3$.
\begin{definition}\label{definition:ulrich_x3}
An \emph{Ulrich bundle} is a vector bundle $\mathcal{E}$ such that $h^{*}(\mathcal{E}(-iH)) = 0$ for $i=1,2$.
\end{definition}

Recall the exceptional collections $\mathfrak{E}$ and $^{\vee}\mathfrak{E}$ of Example \ref{example:x3_exceptional_collections}. Without loss of generality, we twist the elements of those exceptional collections by $\mathcal{O}(H)$. The tilting sheaf $T := \oplus_{i=0}^{5} {^{\vee}E_i}$ has an endomorphism algebra $A$ which is isomorphic to the opposite of a bound quiver algebra $\mathbb{C}K_{3,2,1} / J$ of the quiver
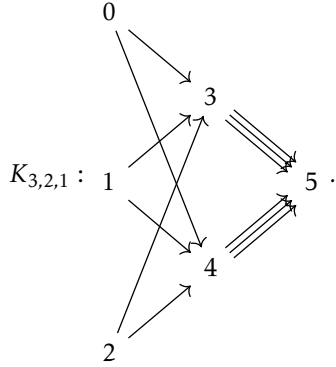
\[
 K_{3,2,1}:
 \begin{tikzcd}
 0
 \arrow[dr]
 \arrow[dddr] & & \\
 & 3
 \arrow[dr]\arrow[shift left, dr]\arrow[shift right, dr] & \\
 1
 \arrow[ur]
 \arrow[dr] & & 5 \\
 & 4
 \arrow[ur]\arrow[shift left, ur]\arrow[shift right, ur] & \\
 2
 \arrow[ur]
 \arrow[uuur] & &
 \end{tikzcd}.
\]
Let $\mathcal{E}$ be an Ulrich bundle of rank $r$. Under the triangulated equivalence $\Phi$ of Theorem \ref{theorem:bondal}, $\mathcal{E}$ is mapped to the complex of quiver representations
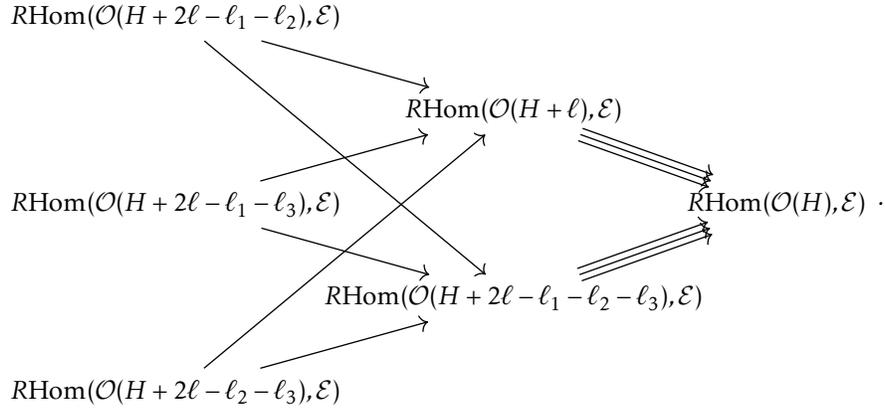
\[
 \begin{tikzcd}[column sep=-1.5em]
 R\mbox{Hom}(\mathcal{O}(H + 2 \ell - \ell_1 - \ell_2),\mathcal{E})
 \arrow[dr]
 \arrow[dddr] & & \\
 & R\mbox{Hom}(\mathcal{O}(H + \ell),\mathcal{E})
 \arrow[dr]\arrow[shift left, dr]\arrow[shift right, dr] & \\
 R\mbox{Hom}(\mathcal{O}(H + 2 \ell - \ell_1 - \ell_3),\mathcal{E})
 \arrow[ur]
 \arrow[dr] & & R\mbox{Hom}(\mathcal{O}(H),\mathcal{E}) \\
 & R\mbox{Hom}(\mathcal{O}(H + 2 \ell - \ell_1 - \ell_2 - \ell_3),\mathcal{E})
 \arrow[ur]\arrow[shift left, ur]\arrow[shift right, ur] & \\
 R\mbox{Hom}(\mathcal{O}(H + 2 \ell - \ell_2 - \ell_3),\mathcal{E})
 \arrow[ur]
 \arrow[uuur] & &
 \end{tikzcd}.
\]
The term on the right vanishes by Definition \ref{definition:ulrich_x3}.

\begin{proposition}
The left and middle terms in the above quiver are all concentrated in the first cohomologies.
\end{proposition}
\begin{proof}
Let $D$ be one of the divisors $\ell$, $2 \ell - \ell_1 - \ell_2 - \ell_3$, $2 \ell - \ell_1 - \ell_2$, $2 \ell - \ell_1 - \ell_3$, $2 \ell - \ell_2 - \ell_3$. Note that $D$ is globally generated by \cite{DiRocco1996}*{Corollary 4.6}, hence effective. Since $\mbox{h}^0(\mathcal{E}(-H))=0$ by Definition \ref{definition:ulrich_x3}, we have $\mbox{h}^0(\mathcal{E}(-H-D))=0$.

We now note that
\[
 \mbox{h}^2(\mathcal{E}(-H-D)) = \mbox{h}^0(\mathcal{E}^{\vee}(K+H+D)) = \mbox{h}^0(\mathcal{E}^{\vee}(3H+K)(-H)(-H+D)).
\]
By Lemma \ref{lemma:ulrich_duality}, $\mathcal{E}^{\vee}(3H+K)$ is an Ulrich bundle and hence $\mbox{h}^0(\mathcal{E}^{\vee}(3H+K)(-H))=0$. The proof will be complete once we show that $H-D$ is an effective divisor. We do this for $D = \ell$; the proof for the other cases is similar.

By Theorem \ref{theorem:ampleness}, $b \leq -1$, $c \leq -1$, $d \leq -1$, $a+b+c \geq 1$, $a+b+d \geq 1$, and $a+c+d \geq 1$. Then, $H - \ell = (a-1) \ell + b \ell_1 + c \ell_2 + d \ell_3$, and we have the inequalities $b \leq 0$, $c \leq 0$, $d \leq 0$, $(a-1)+b+c \geq 0$, $(a-1)+b+d \geq 0$, and $(a-1)+c+d \geq 0$. It follows by \cite{DiRocco1996}*{Corollary 4.6} that $H - \ell$ is globally generated, hence effective.
\end{proof}

Let us denote the dimensions of the remaining terms by $\alpha$, $\beta$, $\gamma$ for the left terms (top to bottom), and by $\delta$, $\epsilon$ for the middle terms (top to bottom). We therefore have a single representation of a bound quiver $(K_{3,2,1};J)$ with dimension vector $\underline{d} = (\alpha, \beta, \gamma, \delta, \epsilon, 0)$, which we consider as a representation of the quiver $K_{3,2}$ with dimension vector $\underline{d} = (\alpha, \beta, \gamma, \delta, \epsilon)$.

The second main result of this article is the following theorem.
\begin{theorem}\label{theorem:ulrich_trichotomy_for_x3}
Ulrich trichotomy holds for $X_3$.
\end{theorem}
\begin{proof}
If there are no Ulrich bundles, $(X_3, \mathcal{O}(H))$ is of finite Ulrich representation type.

Now let us assume that there are Ulrich bundles. Recall that the quiver $K_{3,2}$ is hyperbolic by Remark \ref{remark:hyperbolic_quivers}.

\emph{Case 1.} There exists an Ulrich bundle $\mathcal{E}$, which we assume to be indecomposable without loss of generality by Remark \ref{remark:ulrich}, such that the dimension vector $\underline{d} = (\alpha, \beta, \gamma, \delta, \epsilon)$ of the associated representation of the quiver $K_{3,2}$ satisfies $\chi(\underline{d}, \underline{d}) < 0$. In this case, $(X_3, \mathcal{O}(H))$ is of wild Ulrich representation type. The argument follows that of Theorem \ref{theorem:ulrich_trichotomy_for_p2}, so we omit the details.

\emph{Case 2.} There exists no Ulrich bundle with dimension vector $\underline{d} = (\alpha, \beta, \gamma, \delta, \epsilon)$ of the associated representation of the quiver $K_{3,2}$ such that $\chi(\underline{d}, \underline{d}) < 0$, but there exists an indecomposable Ulrich bundle $\mathcal{E}$ such that the associated representation of the quiver $K_{3,2}$ has dimension vector $\underline{d} = (\alpha, \beta, \gamma, \delta, \epsilon)$ that satisfies $\chi(\underline{d}, \underline{d}) = 0$. Just as in the proof of Case 1 of Theorem \ref{theorem:ulrich_trichotomy_for_p2}, we have that by \cite{Kac1980}*{Proposition 1.6 (a)}, $\underline{d}$ is an imaginary root, and then by \cite{Kac1980}*{Lemma 2.1 (e)} and \cite{Kac1980}*{Lemma 2.7}, $\underline{d}$ is a Schur root. The dimension of the moduli space of these Ulrich bundles is given by
\[
 1 - \chi(\underline{d}, \underline{d}) = 1.
\]
Again as in the proof of Theorem \ref{theorem:ulrich_trichotomy_for_p1_times_p1}, we must show that for any given rank $r$, there are finitely many such $\underline{d}$. In this case, $(X_3, \mathcal{O}(H))$ is of tame Ulrich representation type. The details are given in the proof of Case 3 below.

\emph{Case 3.} There exists no Ulrich bundle with dimension vector $\underline{d} = (\alpha, \beta, \gamma, \delta, \epsilon)$ of the associated representation of the quiver $K_{3,2}$ such that $\chi(\underline{d}, \underline{d}) \leq 0$, but there exists an Ulrich bundle $\mathcal{E}$ of rank $r$ such that the associated representation of the quiver $K_{3,2}$ has dimension vector $\underline{d} = (\alpha, \beta, \gamma, \delta, \epsilon)$ that satisfies $\chi(\underline{d}, \underline{d}) > 0$. Note that $\mathcal{E}$ can again be assumed to be indecomposable without loss of generality. Just as in the proof of Theorem \ref{theorem:ulrich_trichotomy_for_p2}, by \cite{Kac1980}*{Proposition 1.9 (b)}, $\underline{d}$ is a real root, and by \cite{Kac1980}*{Proposition 1.6 (b)}, $\chi(\underline{d}, \underline{d}) = 1$. Again, by \cite{Kac1980}*{Lemma 2.1 (e)} and \cite{Kac1980}*{Theorem 2 (c)}, $\mathcal{E}$ is the unique indecomposable Ulrich bundle with dimension vector equal to $\underline{d}$ of the associated representation of the quiver $K_{3,2}$; and the proof will be complete if we show that there are only finitely many such $\underline{d}$ for the given rank $r$.

By Remark \ref{remark:phi}, the Ulrich bundle $\mathcal{E}$ fits into a short exact sequence of the form
\[
 0 \to \mathcal{O}(\ell_1 - \ell+H)^{\alpha} \oplus \mathcal{O}(\ell_2 - \ell+H)^{\beta} \oplus \mathcal{O}(\ell_3 - \ell+H)^{\gamma} \to \mathcal{P}(H)^{\delta} \oplus \mathcal{K}(H)^{\epsilon} \to \mathcal{E} \to 0.
\]
It follows that
\begin{equation}\label{equation:x3_1}
 2 (\delta + \epsilon) - \alpha - \beta - \gamma = r.
\end{equation}
Next, by Definition \ref{definition:ulrich_x3}, we must have $\chi(\mathcal{E}(-2H))=0$. Twisting the short exact sequence above by $\mathcal{O}(-2H)$ and calculating Euler characteristics, we have
\begin{equation}\label{equation:x3_2}
 A \alpha + B \beta + C \gamma = D \delta + E \epsilon,
\end{equation}
where
\begin{align*}
 A = \chi(\mathcal{O}(\ell_1 - \ell - H)) &= \frac{1}{2}(a^2-a+b-b^2-c-c^2-d-d^2), \\
 B = \chi(\mathcal{O}(\ell_2 - \ell - H)) &= \frac{1}{2}(a^2-a-b-b^2+c-c^2-d-d^2), \\
 C = \chi(\mathcal{O}(\ell_3 - \ell - H)) &= \frac{1}{2}(a^2-a-b-b^2-c-c^2+d-d^2), \\
 D = \chi(\mathcal{P} (-H)) &= -2a + a^2 -b - b^2 - c - c^2 -d - d^2, \\
 E = \chi(\mathcal{K} (-H)) &= -a + a^2 - b^2 - c^2 - d^2.
\end{align*}
Finally, the equation $\chi(\underline{d}, \underline{d}) = 1$ gives
\begin{equation}\label{equation:x3_3}
 \alpha^2 + \beta^2 + \gamma^2 + \delta^2 + \epsilon^2 - \alpha \delta - \alpha \epsilon - \beta \delta - \beta \epsilon - \gamma \delta - \gamma \epsilon = 1
\end{equation}
by Example \ref{example:k32_euler_ringel_form}. Now, if these three equations have finitely many solutions, then $(X_3, \mathcal{O}(H))$ is of either finite Ulrich representation type or tame Ulrich representation type.

Suppose then that these three equations have infinitely many solutions. We may consider these solutions to be in $\mathbb{R}^5$, and hence in real projective space $\mathbb{R} \mathbb{P}^5$, with the extra coordinate given by $z$. Since these solutions form an infinite set in $\mathbb{R} \mathbb{P}^5$, and $\mathbb{R} \mathbb{P}^5$ is a compact topological space, the set of solutions must have an accumulation point, which will then necessarily satisfy $z=0$. Consider such an accumulation point $[\alpha_0 : \beta_0 : \gamma_0 : \delta_0 : \tau_0 : 0]$. It is then clear that $\alpha_0$, $\beta_0$, $\gamma_0$, $\delta_0$, and $\tau_0$ satisfy the equations
\begin{equation}\label{equation:x3_4}
 2 (\delta_0 + \epsilon_0) = \alpha_0 + \beta_0 + \gamma_0,
\end{equation}
\begin{equation}\label{equation:x3_5}
 \alpha_0^2 + \beta_0^2 + \gamma_0^2 + \delta_0^2 + \epsilon_0^2 - \alpha_0 \delta_0 - \alpha_0 \epsilon_0 - \beta_0 \delta_0 - \beta_0 \epsilon_0 - \gamma_0 \delta_0 - \gamma_0 \epsilon_0 = 0,
\end{equation}
and using equations \ref{equation:x3_2} and \ref{equation:x3_4}, the equation
\begin{equation}\label{equation:x3_6}
 (b - c - d) \alpha_0 + (-b + c - d) \beta_0 + (-b - c  + d) \gamma_0 + 2 (a+b+c+d) \delta_0 = 0
\end{equation}
as well. Using equation \ref{equation:x3_4}, equation \ref{equation:x3_5} can be written in the form
\begin{equation}\label{equation:x3_7}
 \alpha_0^2 + \beta_0^2 + \gamma_0^2 + \delta_0^2 + \epsilon_0^2 - 2 (\delta_0 + \epsilon_0)^2 = 0.
\end{equation}
If we now make the transformations
\begin{align*}
 \alpha_0' &= \alpha_0 \\
 \beta_0' &= \beta_0 \\
 \gamma_0' &= \gamma_0 \\
 \delta_0' &= \frac{\delta_0 - \epsilon_0}{\sqrt{2}} \\
 \epsilon_0' &= \sqrt{\frac{3}{2}} (\delta_0 + \epsilon_0),
\end{align*}
equation \ref{equation:x3_7} takes the form
\begin{equation}\label{equation:x3_9}
 (\alpha_0')^2 + (\beta_0')^2 + (\gamma_0')^2 + (\delta_0')^2 - (\epsilon_0')^2 = 0,
\end{equation}
and equation \ref{equation:x3_6} takes the form
\begin{equation}\label{equation:x3_10}
 (b - c - d) \alpha_0' + (-b + c - d) \beta_0' + (-b - c  + d) \gamma_0' + (a+b+c+d) (\sqrt{2} \delta_0' + \sqrt{\frac{2}{3}} \epsilon_0') = 0.
\end{equation}
We use equation \ref{equation:x3_4} to eliminate $\epsilon_0'$, which gives
\begin{equation}\label{equation:x3_11}
 (a+3b-c-d) \alpha_0' + (a-b+3c-d) \beta_0' + (a-b-c+3d) \gamma_0' + 2 \sqrt{2} (a+b+c+d) \delta_0' = 0.
\end{equation}
Using equation \ref{equation:x3_4}, equation \ref{equation:x3_9} can be further simplified to
\begin{equation}\label{equation:x3_12}
 (\alpha_0')^2 + (\beta_0')^2 + (\gamma_0')^2 + (\delta_0')^2 - \frac{3}{8} (\alpha_0' + \beta_0' + \gamma_0')^2 = 0.
\end{equation}

We make another set of transformations
\begin{align*}
 \alpha_0'' &= \frac{1}{\sqrt{3}} (\alpha_0' + \beta_0' + \gamma_0') \\
 \beta_0'' &= 2 (- \alpha_0' + \beta_0') \\
 \gamma_0'' &= \frac{2}{\sqrt{3}} (- \alpha_0' - \beta_0' + 2 \gamma_0') \\
 \delta_0'' &= 2\sqrt{2} \delta_0'
\end{align*}
to convert equation \ref{equation:x3_12} to
\begin{equation}\label{equation:x3_13}
 - (\alpha_0'')^2 + (\beta_0'')^2 + (\gamma_0'')^2 + (\delta_0'')^2 = 0.
\end{equation}
After this set of transformations, equation \ref{equation:x3_11} takes the form
\begin{equation}\label{equation:x3_14}
 \frac{3a+b+c+d}{\sqrt{3}} \alpha_0'' + (-b+c) \beta_0'' + \frac{-b-c+2d}{\sqrt{3}} \gamma_0'' + (a+b+c+d) \delta_0'' = 0.
\end{equation}

The proof will now be complete once we show that equations \ref{equation:x3_13} and \ref{equation:x3_14} have the unique solution $(0,0,0,0)$.

Equation \ref{equation:x3_13} describes a cone in $\mathbb{R}^4$ with the $\alpha_0''$-axis as its axis. Equation \ref{equation:x3_14} describes a hyperplane in $\mathbb{R}^4$ with normal vector
\begin{equation*}
 \mathbf{n} = (p,q,r,s) = (\frac{3a+b+c+d}{\sqrt{3}}, -b+c, \frac{-b-c+2d}{\sqrt{3}}, a+b+c+d).
\end{equation*}
It is enough to show that the angle $\theta$ between $\mathbf{n}$ and the unit vector $\mathbf{u}$ in the direction of the $\alpha_0''$-axis is smaller than 45\textdegree{}, which is equivalent to
\[
 \abs{\cos \theta} = \frac{\abs{\mathbf{n} \boldsymbol{\cdot} \mathbf{u}}}{\abs{\mathbf{n}}} > \frac{1}{\sqrt{2}},
\]
and hence, to
\begin{equation*}
 p^2 > q^2 + r^2 + s^2 .
\end{equation*}
Substituting the expressions for $p$,$q$, $r$, and $s$ into this inequality, we get the new inequality
\begin{equation*}
 a^2 - b^2 - c^2 - d^2 > 0.
\end{equation*}
Since $H$ is ample, this inequality is satisfied.
\end{proof}

\subsection{The Del Pezzo Quintic $X_4$}

We come to the third main result of this article, namely that Conjecture \ref{conjecture:trichotomy} holds for $X_4$.

Throughout this subsection, all sheaves and cohomologies are considered on $X_4$ without explicit mention. Also, Ulrich bundles are considered with respect to an arbitrary but fixed polarization $\mathcal{O}(H)$ where $H = a \ell + b \ell_1 + c \ell_2 + d \ell_3 + e \ell_4$. We note that $\mathcal{O}(H)$ is both ample and very ample by Theorem \ref{theorem:ampleness}.

We restate Definition \ref{definition:ulrich} for vector bundles on $X_4$.
\begin{definition}\label{definition:ulrich_x4}
An \emph{Ulrich bundle} is a vector bundle $\mathcal{E}$ such that $h^{*}(\mathcal{E}(-iH)) = 0$ for $i=1,2$.
\end{definition}

Recall the exceptional collections $\mathfrak{E}$ and $^{\vee}\mathfrak{E}$ of Example \ref{example:x4_exceptional_collections}. Without loss of generality, we twist the elements of those exceptional collections by $\mathcal{O}(H)$. The tilting sheaf $T := \oplus_{i=0}^{6} {^{\vee}E_i}$ has an endomorphism algebra $A$ which is isomorphic to the opposite of a bound quiver algebra $\mathbb{C}K_{5,1,1} / J$ of the quiver
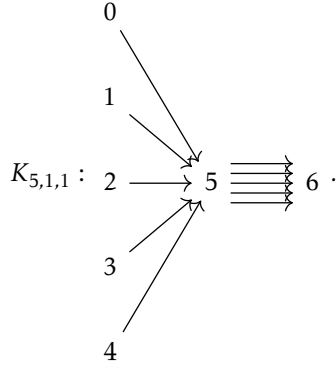
\[
 K_{5,1,1}:
 \begin{tikzcd}
 0\arrow[ddr] \\
 1\arrow[dr] \\
 2\arrow[r] & 5\arrow[r]\arrow[shift left=1.5, r]\arrow[shift right=1.5, r]\arrow[shift left=3, r]\arrow[shift right=3, r] & 6 \\
 3\arrow[ur] \\
 4\arrow[uur]
 \end{tikzcd}.
\]
Let $\mathcal{E}$ be an Ulrich bundle of rank $r$. Under the triangulated equivalence $\Phi$ of Theorem \ref{theorem:bondal}, $\mathcal{E}$ is mapped to the complex of quiver representations
\[
 \begin{tikzcd}[column sep=-1.8em]
 & R\mbox{Hom}(\mathcal{O}(H+2\ell - \ell_1 - \ell_2 - \ell_3),\mathcal{E})\arrow[dd] & \\
 R\mbox{Hom}(\mathcal{O}(H+2\ell - \ell_1 - \ell_2 - \ell_4),\mathcal{E})\arrow[dr] & & \\
 R\mbox{Hom}(\mathcal{O}(H+2\ell - \ell_1 - \ell_3 - \ell_4),\mathcal{E})\arrow[r] & R\mbox{Hom}(\mathcal{F}(H),\mathcal{E})\arrow[r]\arrow[shift left=1.5, r]\arrow[shift right=1.5, r]\arrow[shift left=3, r]\arrow[shift right=3, r] & R\mbox{Hom}(\mathcal{O}(H),\mathcal{E}) \\
 R\mbox{Hom}(\mathcal{O}(H+2\ell - \ell_2 - \ell_3 - \ell_4),\mathcal{E})\arrow[ur] & & \\
 & R\mbox{Hom}(\mathcal{O}(H+\ell),\mathcal{E})\arrow[uu] &
 \end{tikzcd}.
\]
The term on the right vanishes by Definition \ref{definition:ulrich_x4}.

\begin{proposition}
The left and middle terms in the above quiver are all concentrated in the first cohomologies.
\end{proposition}
\begin{proof}
Let $D$ be one of the divisors $\ell$, $2\ell - \ell_2 - \ell_3 - \ell_4$, $2\ell - \ell_1 - \ell_3 - \ell_4$, $2\ell - \ell_1 - \ell_2 - \ell_4$, $2\ell - \ell_1 - \ell_2 - \ell_3$, and finally $2\ell-\ell_1-\ell_2-\ell_3-\ell_4$. (The last one will be needed below.) Note that $D$ is globally generated by \cite{DiRocco1996}*{Corollary 4.6}, hence effective. Since $\mbox{h}^0(\mathcal{E}(-H))=0$ by Definition \ref{definition:ulrich_x4}, we have $\mbox{h}^0(\mathcal{E}(-H-D))=0$.

We now note that
\[
 \mbox{h}^2(\mathcal{E}(-H-D)) = \mbox{h}^0(\mathcal{E}^{\vee}(K+H+D)) = \mbox{h}^0(\mathcal{E}^{\vee}(3H+K)(-H)(-H+D)).
\]
By Lemma \ref{lemma:ulrich_duality}, $\mathcal{E}^{\vee}(3H+K)$ is an Ulrich bundle and hence $\mbox{h}^0(\mathcal{E}^{\vee}(3H+K)(-H))=0$. Once we show that $H-D$ is an effective divisor, we shall obtain $\mbox{h}^2(\mathcal{E}(-H-D))=0$. We do this for $D = \ell$; the proof for the other cases is similar.

By Theorem \ref{theorem:ampleness}, we have $b \leq -1$, $c \leq -1$, $d \leq -1$, $e \leq -1$, $a+b+c \geq 1$, $a+b+d \geq 1$, $a+b+e \geq 1$, $a+c+d \geq 1$, $a+c+e \geq 1$, $a+d+e \geq 1$. Then, $H - \ell = (a-1) \ell + b \ell_1 + c \ell_2 + d \ell_3 + e \ell_4$, and we have the inequalities $b \leq 0$, $c \leq 0$, $d \leq 0$, $e \leq 0$, $(a-1)+b+c \geq 0$, $(a-1)+b+d \geq 0$, $(a-1)+b+e \geq 0$, $(a-1)+c+d \geq 0$, $(a-1)+c+e \geq 0$, $(a-1)+d+e \geq 0$. It follows by \cite{DiRocco1996}*{Corollary 4.6} that $H - \ell$ is globally generated, hence effective.

It now remains to prove that $R\mbox{Hom}(\mathcal{F}(H),\mathcal{E})$ is concentrated in the first cohomology. Consider the short exact sequence
\[
 0 \to \mathcal{O}(-\ell) \to \mathcal{F}^{\vee} \to \mathcal{O}(-2\ell + \ell_1 + \ell_2 + \ell_3 + \ell_4) \to 0,
\]
and twist it by $\mathcal{E}(-H)$ to obtain
\[
 0 \to \mathcal{E}(-H-\ell) \to \mathcal{E}(-H) \otimes \mathcal{F}^{\vee} \to \mathcal{O}(-H-2\ell + \ell_1 + \ell_2 + \ell_3 + \ell_4) \to 0.
\]
Only the first cohomologies of the left and the right terms are possibly nonzero, as was shown above. Hence the same holds for $\mathcal{E}(-H) \otimes \mathcal{F}^{\vee}$.
\end{proof}

Let us denote the dimensions of the remaining terms by $\alpha$, $\beta$, $\gamma$, $\delta$, $\epsilon$ for the terms on the top, on the left (top to bottom), and on the bottom, respectively. Let $\zeta$ denote the dimension of the term in the center. We therefore have a single representation of a bound quiver $(K_{5,1,1};J)$ with dimension vector $\underline{d} = (\alpha, \beta, \gamma, \delta, \epsilon, \zeta, 0)$, which we consider as a representation of the quiver $K_{5,1}$ with dimension vector $\underline{d} = (\alpha, \beta, \gamma, \delta, \epsilon, \zeta)$.

The third main result of this article is the following theorem.
\begin{theorem}\label{theorem:ulrich_trichotomy_for_x4}
Ulrich trichotomy holds for $X_4$.
\end{theorem}
\begin{proof}
Since the beginning of the proof is similar to the proof of Theorem \ref{theorem:ulrich_trichotomy_for_x3}, we skip directly into the part in Case 3 where the Ulrich bundle is written as part of a short exact sequence and the equations derived. Recall that the quiver $K_{5,1}$ is hyperbolic by Remark \ref{remark:hyperbolic_quivers}.

We suppose that there exists an Ulrich bundle $\mathcal{E}$ of rank $r$, assumed to be indecomposable without loss of generality, such that the associated representation of the quiver $K_{5,1}$ has dimension vector $\underline{d} = (\alpha, \beta, \gamma, \delta, \epsilon, \zeta)$ such that $\chi(\underline{d}, \underline{d}) > 0$. Note that we have $\chi(\underline{d}, \underline{d}) = 1$ in this case, as in the proof of Case 3 of Theorem \ref{theorem:ulrich_trichotomy_for_x3}. We want to show that there are only finitely many possible $\underline{d}$ for a given rank $r$.

By Remark \ref{remark:phi}, the Ulrich bundle $\mathcal{E}$ fits into a short exact sequence of the form
\[
 0 \to \mathcal{S} \oplus \mathcal{O}(-2\ell+\ell_1+\ldots+\ell_4+H)^{\epsilon} \to \mathcal{K}(H)^{\zeta} \to \mathcal{E} \to 0,
\]
where $\mathcal{S} = \mathcal{O}(\ell_4 - \ell+H)^{\alpha} \oplus \mathcal{O}(\ell_3 - \ell+H)^{\beta} \oplus \mathcal{O}(\ell_2 - \ell+H)^{\gamma} \oplus \mathcal{O}(\ell_1 - \ell+H)^{\delta}$. From the ranks, we get the equation
\begin{equation}
 3\zeta - \alpha - \beta - \gamma - \delta - \epsilon = r.
\end{equation}
Next, by Definition \ref{definition:ulrich_x4}, we have $\chi(\mathcal{E}(-2H))=0$. Twisting the short exact sequence above by $\mathcal{O}(-2H)$ and calculating Euler characteristics, we have
\begin{equation}\label{equation:x4_1}
 Z\zeta = A\alpha + B\beta + C\gamma + D\delta + E\epsilon,
\end{equation}
where
\begin{align*}
 A = \chi(\mathcal{O}(\ell_4 - \ell - H)) &= \frac{1}{2}(a^2-a-\Sigma-\Sigma_2+e), \\
 B = \chi(\mathcal{O}(\ell_3 - \ell - H)) &= \frac{1}{2}(a^2-a-\Sigma-\Sigma_2+d), \\
 C = \chi(\mathcal{O}(\ell_2 - \ell - H)) &= \frac{1}{2}(a^2-a-\Sigma-\Sigma_2+c), \\
 D = \chi(\mathcal{O}(\ell_1 - \ell - H)) &= \frac{1}{2}(a^2-a-\Sigma-\Sigma_2+b), \\
 E = \chi(\mathcal{O}(-2\ell+\ell_1+\ell_2+\ell_3+\ell_4 - H)) &= \frac{1}{2}(a^2+a+\Sigma-\Sigma_2), \\
 Z = \chi(\mathcal{K} (-H)) &= \frac{1}{2}(3a^2-3a-\Sigma-3\Sigma_2), 
\end{align*}
where $\Sigma=b+c+d+e$ and $\Sigma_2 = b^2 + c^2 + d^2 + e^2$.
Finally, the equation $\chi(\underline{d}, \underline{d}) = 1$ gives
\begin{equation}
 \alpha^2 + \beta^2 + \gamma^2 + \delta^2 + \epsilon^2 + \zeta^2 - \zeta \alpha - \zeta \beta - \zeta \gamma - \zeta \delta - \zeta \epsilon = 1
\end{equation}
by Example \ref{example:k51_euler_ringel_form}. We want to show that these three equations have finitely many solutions.

Suppose then that these three equations have infinitely many solutions. We may consider these solutions to be in $\mathbb{R}^6$, and hence in real projective space $\mathbb{R} \mathbb{P}^6$, with the extra coordinate given by $z$. Since these solutions form an infinite set in $\mathbb{R} \mathbb{P}^6$, and $\mathbb{R} \mathbb{P}^6$ is a compact topological space, the set of solutions must have an accumulation point, which will then necessarily satisfy $z=0$. Consider such an accumulation point $[\alpha_0 : \beta_0 : \gamma_0 : \delta_0 : \epsilon_0 : \zeta_0 : 0]$. It is then clear that $\alpha_0$, $\beta_0$, $\gamma_0$, $\delta_0$, $\epsilon_0$, and $\zeta_0$ satisfy the equations
\begin{equation}\label{equation:x4_2}
 3\zeta_0 = \alpha_0 + \beta_0 + \gamma_0 + \delta_0 + \epsilon_0,
\end{equation}
\begin{equation}\label{equation:x4_3}
 \alpha_0^2 + \beta_0^2 + \gamma_0^2 + \delta_0^2 + \epsilon_0^2 + \zeta_0^2 = \zeta_0 (\alpha_0 + \beta_0 + \gamma_0 + \delta_0 + \epsilon_0),
\end{equation}
and using equations \ref{equation:x4_1} and \ref{equation:x4_2}, the equation
\begin{equation}\label{equation:x4_4}
\begin{split}
 (-b-c-d+2e)\alpha_0 &+ (-b-c+2d-e)\beta_0 + (-b+2c-d-e)\gamma_0 \\
 &+ (2b-c-d-e)\delta_0 + (3a+2\Sigma)\epsilon_0 = 0
\end{split}
\end{equation}
as well. Using equation \ref{equation:x4_2}, equation \ref{equation:x4_3} can be written in the form
\begin{equation}\label{equation:x4_5}
 9(\alpha_0^2 + \beta_0^2 + \gamma_0^2 + \delta_0^2 + \epsilon_0^2)-2(\alpha_0 + \beta_0 + \gamma_0 + \delta_0 + \epsilon_0)^2 = 0.
\end{equation}
If we now make the transformations
\begin{align*}
 \alpha_0' &= \frac{3}{\sqrt{2}}(-\alpha_0 + \epsilon_0) \\
 \beta_0' &= \sqrt{\frac{3}{2}}(-\alpha_0+2\delta_0-\epsilon_0) \\
 \gamma_0' &= \frac{\sqrt{3}}{2}(-\alpha_0+3\gamma_0-\delta_0-\epsilon_0) \\
 \delta_0' &= \frac{3}{2\sqrt{5}}(-\alpha_0+4\beta_0-\gamma_0-\delta_0-\epsilon_0) \\
 \epsilon_0' &= \frac{1}{\sqrt{5}}(\alpha_0+\beta_0+\gamma_0+\delta_0+\epsilon_0),
\end{align*}
equation \ref{equation:x4_5} takes the form
\begin{equation}\label{equation:x4_6}
 (\alpha_0')^2 + (\beta_0')^2 + (\gamma_0')^2 + (\delta_0')^2 - (\epsilon_0')^2 = 0,
\end{equation}
and equation \ref{equation:x4_4} takes the form
\begin{equation}\label{equation:x4_7}
\begin{split}
 & \frac{1}{\sqrt{2}}(a+b+c+d)\alpha_0' + \frac{1}{\sqrt{6}}(-a+b-c-d-2e)\beta_0' + \\
 & \frac{1}{2\sqrt{3}}(-a-2b+2c-d-2e)\gamma_0' + \frac{1}{2\sqrt{5}}(-a-2b-2c+3d-2e)\delta_0' + \\
 & \frac{1}{\sqrt{5}}(3a+b+c+d+e)\epsilon_0' = 0.
\end{split}
\end{equation}
The proof will now be complete once we show that equations \ref{equation:x4_6} and \ref{equation:x4_7} have the unique solution $(0,0,0,0,0)$.

Equation \ref{equation:x4_6} describes a cone in $\mathbb{R}^5$ with the $\epsilon_0'$-axis as its axis. Equation \ref{equation:x4_7} describes a hyperplane in $\mathbb{R}^5$ with normal vector $\mathbf{n} = (p,q,r,s,t)$, where
\begin{align*}
 p &= \frac{1}{\sqrt{2}}(a+b+c+d), \\
 q &= \frac{1}{\sqrt{6}}(-a+b-c-d-2e), \\
 r &= \frac{1}{2\sqrt{3}}(-a-2b+2c-d-2e), \\
 s &= \frac{1}{2\sqrt{5}}(-a-2b-2c+3d-2e), \\
 t &= \frac{1}{\sqrt{5}}(3a+b+c+d+e).
\end{align*}
It is enough to show that the angle $\theta$ between $\mathbf{n}$ and the unit vector $\mathbf{u}$ in the direction of the $\epsilon_0'$-axis is smaller than 45\textdegree{}, which is equivalent to
\[
 \abs{\cos \theta} = \frac{\abs{\mathbf{n} \boldsymbol{\cdot} \mathbf{u}}}{\abs{\mathbf{n}}} > \frac{1}{\sqrt{2}},
\]
and hence, to
\begin{equation*}
 t^2 > p^2 + q^2 + r^2 + s^2.
\end{equation*}
Substituting the expressions for $p$,$q$, $r$, $s$, and $t$ into this inequality, we get the new inequality
\begin{equation*}
 a^2 - b^2 - c^2 - d^2 - e^2 > 0.
\end{equation*}
Since $H$ is ample, this inequality is satisfied.
\end{proof}


\end{document}